\documentclass[11pt]{amsart}
\usepackage{a4wide}

\usepackage{amssymb}
\usepackage{amsmath}

\usepackage{t1enc}
\usepackage[latin2]{inputenc}
\usepackage[english]{babel}
\usepackage{verbatim}
\usepackage{color}
\usepackage[dvips]{graphicx}

\theoremstyle{plain}
\newtheorem{theorem}{Theorem}[section]
\newtheorem{lemma}[theorem]{Lemma}
\newtheorem{proposition}[theorem]{Proposition}

\newtheorem{corollary}[theorem]{Corollary}
\newtheorem*{problem}{Problem}
\theoremstyle{definition}
\newtheorem{definition}[theorem]{Definition}

\begin{document}

\newcommand{\R}{{\mathbb G}}
\newcommand{\pp}{\mathbb{P}}

\renewcommand{\P}{\mathcal P}
\newcommand{\Q}{\mathcal Q}
\renewcommand{\S}{\mathcal S}

\newcommand{\oo}{\mathcal{O}}
\newcommand{\qq}{\mathcal{Q}}
\newcommand{\C}{\mathcal{C}}
\newcommand{\D}{\mathcal{D}}

\newcommand{\sw}{\operatorname{Sw}}
\newcommand{\rev}{\operatorname{Rev}}
\newcommand{\sym}{\operatorname{Sym}}
\newcommand{\aut}{\operatorname{Aut}}
\newcommand{\Max}{\operatorname{Max}}

\newcommand{\Ro}{\mathfrak R}
\newcommand{\Cu}{\mathfrak C}

\title{A new operation on partially ordered sets}

\author[P. P. Pach]{P\'{e}ter P\'{a}l Pach }
\address[1,4]{E\"{o}tv\"{o}s Lor\'{a}nd University, 
          Department of Algebra and Number Theory,
         1117 Budapest,  P\'{a}zm\'{a}ny P\'{e}ter s\'{e}t\'{a}ny 1/c, Hungary}

\email{ppp24@cs.elte.hu}
\email{csaba@cs.elte.hu}

\author[M.~Pinsker]{Michael Pinsker}
   \address[2]{\'{E}quipe de Logique Math\'{e}matique\\ Universit\'{e}
Diderot -- Paris 7\\
       UFR de Math\'{e}matiques\\
       75205 Paris Cedex 13, France}
   \email{marula@gmx.at}
   \urladdr{http://dmg.tuwien.ac.at/pinsker/}
   \thanks{Research of the second author supported by an
APART-fellowship of the Austrian Academy of Sciences.}

\author[A. Pongr\'{a}cz]{Andr\'{a}s Pongr\'{a}cz }
\address[3]{Central European University, Budapest, Hungary}
\email{pongeee@cs.elte.hu}

\author[Cs. Szab\'{o}]{Csaba Szab\'{o}}

\pagebreak

\keywords{poset; rotation; random poset; Fra\"{i}ss\'{e} class; reduct; Seidel switch; graph isomorphism}
\date{\today}

\begin{abstract}
Recently it has been shown that all non-trivial closed permutation
groups containing the automorphism group of the random poset are generated by two types of permutations: the first type are permutations turning
the order upside down, and the second type are permutations induced by so-called
rotations. In this paper we introduce rotations for finite
posets, which can be seen as the poset counterpart of Seidel-switch
 for finite graphs. We analyze some of their combinatorial properties, and investigate in particular the question of when two finite posets are 
rotation-equivalent. We moreover give an explicit combinatorial 
construction of a rotation of the random poset whose image is again isomorphic to the random poset. As an corollary of our results on rotations of finite posets, we obtain that the group of rotating permutations of the random poset is the automorphism group of a homogeneous structure in a finite language.
\end{abstract}

\maketitle

\section{Operations from generic objects}
 Switching of a graph  was introduced by van Lint and
Seidel   in connection with a problem of finding
equilateral $n$-tuples of points in elliptic geometry in 1966  \cite{LiSe}.
The operation of \emph{switch} or  \emph{Seidel-switch} on a graph  with respect to
a set $X$ of vertices works as
follows: interchange edges and non-edges between
$X$ and its complement, leaving edges within and
outside $X$ unaltered. Two graphs are called \emph{switch-equivalent} iff one
can be obtained from the other one up to isomorphism by application of a Seidel-switch \cite{seidel1}.

Since then the switch operation on graphs has
gained applications in several areas of mathematics.
In   
\cite{Ser} the Seidel-switch is applied
 to construct  cospectral non-isomorphic graphs; for more details on the   role of the Seidel-switch in the
theory of spectra of graphs see the monograph \cite{BrA}.
 Switches are also used in 
 the classification of root systems and Weyl-groups \cite{CaSe}, or
 in coding theory  \cite{Krat}. 
 In \cite{JoKl}
 new strongly regular and amorphic association schemes are constructed
 using switch operations. In  geometry  isomorphic invariants
for translation planes are defined via the Seidel-switch  \cite{Moor}.
Several authors have considered the complexity of deciding whether or
not a given 
graph can be switched to a graph 
having some particular property. 
Polynomial-time algorithms are known for switching to a triangle-free
graph \cite{Hay96, HHW02}, to a 
claw-free graph \cite{JK08}, to an Eulerian graph \cite{HHW02}, to a bipartite
graph \cite{HHW02}, and to a planar 
graph \cite{EHHR00a, Kra03}.
On the hard side,
Kratochvil \cite{Kra03} has shown NP-completeness 
of deciding whether or not a given graph can be switched to a regular one.
In  \cite{eop} the complexity of  parameterized problems related to
switching are investigated.
 It is also proved  in \cite{KratNes} that deciding whether or not
two given finite graphs are switch-equivalent is isomorphism-complete. 

A priori the switch operation might look quite arbitrary among all operations that one could imagine on graphs, and its popularity surprising, but its special role is reflected in the following result from \cite{thomas} concerning the \emph{random graph}. The random graph $\R$ \cite{erdr} is the unique countably infinite graph which is  \emph{universal} in the sense that any finite  graph is isomorphic with an induced subgraph of $\R$ and which is \emph{homogeneous},
that is,  any isomorphism between
finite induced subgraphs of   $\R$ extends to an
automorphism of  $\R$. The random graph 
 is obtained as the \emph{Fra\"{\i}ss\'e limit} of the class of all finite
graphs, and can thus be seen as the generic object representing the class
of finite graphs (see e.g. \cite{fra}). Moreover, if the
edges of a countably infinite  
graph are chosen independently with probability $\frac{1}{2}$, then this graph is  isomorphic to  $\R$   with probability 1. It is not difficult to see that if one applies the Seidel-switch to $\R$ with respect to a finite subset of $\R$, then the resulting graph is isomorphic to $\R$. Hence, there is a permutation on $\R$ witnessing this isomorphism; call such permutations \emph{switching}.
The mentioned result from \cite{thomas}, due to Thomas, states the following.
 
 \begin{theorem}[from~\cite{thomas}]\label{thom} The closed
subgroups of the full symmetric group $\sym (\R)$ on $\R$ which contain the automorphism group $\aut (\R)$ of $\R$ are precisely the following.
\begin{enumerate}
\item $\aut(\R)$;
\item the group of automorphisms and anti-automorphisms of $\R$;
\item the group of switching permutations of $\R$;
\item the join of (2) and (3); 
\item $\sym(\R)$.
\end{enumerate}
\end{theorem}

 This theorem can be interpreted as follows: just like $\aut(\R)$ is the group of all symmetries of $\R$, the closed permutation groups containing $\aut(\R)$ stand for all symmetries of $\R$ if we are willing to give up some of the structure of $\R$. For example, it is obvious that flipping edges and non-edges of $\R$ yields a graph isomorphic with $\R$; this symmetry is reflected by one of the closed groups containing $\aut(\R)$, namely the group of all automorphisms and anti-automorphisms of $\R$. Now the theorem implies that the generic graph $\R$ has only one other symmetry of this kind -- the one reflected by the Seidel-switch.
 
 In this paper we initiate the investigation of the analogue of the
 Seidel-switch for partial orders. Similarly to the situation with
 graphs, there exists a generic partial order $\pp$, called the
 \emph{random partial order}, which represents the class of finite
 partial orders. More precisely, $\pp$ is the Fra\"{\i}ss\'e limit of
 the class of all finite partial orders, or the unique countable
 partial order which is homogeneous and universal for the class of
 finite partial orders. 
In the companion paper \cite{ppppsz} it is shown that this generic order has only two symmetries in the above sense: one symmetry is given by reversing the order, and the second one by so-called \emph{rotations} (see Definition~\ref{defrot}).

\begin{theorem}[from \cite{ppppsz}]\label{reductok}   The closed
subgroups of $\sym (\pp)$ containing $\aut (\pp)$ are precisely the following.
\begin{enumerate}
\item $\aut(\pp)$;
\item the group of automorphisms   and anti-automorphisms of $\pp$;
\item the group of rotating permutations  of $\pp$;
\item the join of (2) and (3);
\item $\sym(\pp)$.
\end{enumerate}
\end{theorem}

Any function on the generic partial order yields functions on all finite partial orders by restriction -- in
particular, by restricting rotating permutations of the generic partial order, we obtain a
notion of rotation for finite partial orders. The study of such
rotations of finite posets is worthwhile; the very same combinatorial
questions which have been considered for the Seidel-switch over the years do
make perfect sense for these new operations.  

We introduce the notion of a rotation of a poset and obtain several combinatorial results concerning rotations of finite posets. We first show that rotations can be decomposed into so-called \emph{cuts}, and that rotations of finite posets can be further decomposed into cuts with respect to a single element.
  We then give a  
description of \emph{rotation-equivalent} finite posets via their three-element subsets and maximal elements, and further investigate  the equivalence classes of
rotation equivalence, in particular the 
classes of chains and antichains.  We show that 
 deciding for two given finite posets $\P, \Q$ whether $\P$ is rotation-equivalent to a poset isomorphic with $\Q$ is graph-isomorphism complete. In Section~\ref{sect:randompo} we turn to rotations of $\pp$. Using a one-point extension of $\pp$ we explicitly construct a non-trivial rotation of the random poset $\pp$ whose image is isomorphic to $\pp$, thus showing that group~(3) of Theorem~\ref{reductok} properly contains $\aut(\pp)$. We also show that rotations of finite posets are precisely the  restrictions of rotating permutations of $\pp$. Finally, as an application of our results on rotations of finite posets, we obtain that the group of rotating permutations of $\pp$ is the automorphism group of a homogeneous relational structure with one ternary relation.

\section{Notation}

Before starting our investigations we introduce some notation concerning partially 
ordered sets. Let $(P,\leq)$   be a poset. For $x,y\in P$  we write $x<y$ iff
$x\leq y$ and  $x\neq y $; $x\geq y$ iff $y\leq x$; and $x \perp y$ iff $x\nleq y$ and $y\nleq x$. For $X,Y\subseteq P$  we write $X<Y$ iff for
all 
$x\in X$ and all $y\in Y$ we have $x<y$. Instead of $\{x\}<Y$ we write
$x<Y$. The set of maximal elements of  $(P,\leq)$ is denoted by $\max
(P,\leq)$. For two posets with disjoint domains $P$ and $Q$ let $P\oplus Q$ denote  the \emph{linear
sum} of $P$ and $Q$, i.e., the poset with domain $P\cup Q$ whose order is defined in the
following way:  the respective orders on $P$ and on $Q$ are
kept, and $P<Q$. A subset $X\subseteq P$ is a \emph{downset}
iff for any $x\in X$ and $y\in P$ if $y<x$, then $y\in X$. An
\emph{up-set} is defined dually. When $X\subseteq P$, then we write $(X,\leq)$ for what really is $(X,X^2\cap \leq)$.

\section{Rotations}

\subsection{Defining rotations and cuts} Seidel's switch assigns a graph to a graph.  
 Rotations and cuts  assign a poset to a poset.

\begin{definition}\label{defrot} Let $\P=(P,\leq)$ be a poset, and let $A$ and $C$
  be disjoint subsets of $P$ such that $A$ is a  downset, $C$ is an
  up-set and $A<C$.  Let $B:=P\setminus(A\cup C) $. 
We assign to $\P$ a new structure 
$\Ro_{A,C}(\P)=(P,\preceq)$ by setting  $x\preceq y$ if and only if one of the following hold:
\begin{itemize}
\item $x,y\in A$ or $x,y\in B$ or $x,y\in C$, and $x\leq y$;
\item $x\in B,y\in A$ and $  x\perp y$ in $(P,\leq)$;
\item $x\in C$ and $y\in A$;
\item $x\in C$, $y\in B$ and $ x\perp y$ in $(P,\leq)$.
\end{itemize} 
 The assignment $\Ro_{A,C}$ is called a \emph{rotation}. If $C$
is empty, then $\Ro_{A,\emptyset}=:\Cu_A$ is called a \emph{cut}. If
$A=\{a\}$ for some $a\in P$ then we write $\Cu_{a}$ for $\Cu_{\{a\}}$.

\end{definition}
In what follows, when we write $\Ro_{A,C}$ or $\Cu_A$, we shall
assume that $A$ and $C$ satisfy the conditions of
Definition~\ref{defrot}. 
Let us first observe that under these conditions, rotations of posets yield posets.

\begin{proposition} Let $\P=(P,\leq)$ be a poset, and let $A,C\subseteq P$ be as in Definition~\ref{defrot}. Then $\Ro_{A,C}(\P)$ is a poset. In
  particular, $\Cu_{A}(\P)$ is a poset.
\end{proposition}

\begin{proof}
Reflexivity and antisymmetry of $\preceq $ follow easily from
Definition~\ref{defrot}. To check transitivity suppose that $x\preceq
y\preceq z$. If $x,y$ and $z$ belong to the same set in the
partitioning $P=A\cup B\cup C$, then $\preceq $ and $\leq $ are equal, hence
$x\preceq  z$. 
If ${x\in A}$, then $y,z\in  A$, as $A$ is an up-set in $\Ro_{A,C}(\P)$. Hence
 $x\preceq z$.  
Now suppose that ${ x\in B}$. 
If $z\in B$ as well, then $y\in B $ and     $x\preceq  z$ is proved. 
If $y\in B$ and  $ z\in A$, then  $y\perp z$ in $\P$. Now, $x\geq z$
would imply $y\geq z$ by the transitivity of $\leq$, hence  $x\perp z$ in $\P$,
 so  $x\preceq  z$. If $y, z \in A$, then $x\perp y$ in $\P$, hence
$x\perp z$  as before and  $x\preceq  z$ holds, again.
Finally, let ${x\in C}$. If $z\in C$ or  $z\in A$, then
$x\preceq z$ is trivial. The case $z\in B$ is handled as
the case $x\in B$.

\end{proof}

\subsection{Rotations via one-point extensions}

The following proposition shows that any rotation of a poset $\P$ corresponds to a certain operation on a one-point extension of $\P$.

\begin{proposition}\label{rotext} Let $\P=(P,\leq)$ be a poset.

\begin{itemize}
\item Consider any 
  one-point extension of $\P$ with domain $P\cup\{a\}$. Set $A:=\{x\in P\,|\, x<a     \}$,   $C:=\{x\in P\,|\, x>a
  \}$ and $B:=\{x\in P\,|\, x\perp a \}$. Then the subsets $A,B,C$ of $P$ satisfy the conditions of Definition~\ref{defrot} for $\P$, and so the rotation  $\Ro_{A,C}$ of $\P$ is defined.
  \item
  Conversely, whenever subsets $A,B,C$ of $P$ satisfy the conditions of Definition~\ref{defrot}, and hence the rotation
  $\Ro_{A,C}$ of $\P$ is defined, then these sets can be defined in this way via a
   one-point extension to domain $P\cup\{a\}$ by setting 
  $a<C$, $a>A$ and $a\perp B$.
  \end{itemize}
\end{proposition}
\begin{proof}
The first part of the statement is implied by the fact that $A$ is a
downset, $C$ is an up-set, and that $A<x<C$ implies $A<C$. Hence the conditions
of a rotation are satisfied. For the second part we need to show that
we obtain a poset this way.  To check transitivity we only need to
examine those triples that contain $a$. If $x<y<a$, then $x\in A$,
hence $x<a$. If $x<a<z$, then  $x\in A$ and $z\in C$,
hence by the assumptions $x<z$.  If $a<y<z$, then $z\in C$,
hence $x<z$.
\end{proof}

By the preceding proposition, the ``rotated'' sets $A$, $B$ and $C$ of a rotation of a partial order $(P,\leq)$ can be imagined as being definable in a one-point extension with domain $P\cup\{a\}$ by the parameter $a$. This justifies the following definition.

\begin{definition} Let $(P,\leq)$ be  a poset and $A,B,C$ pairwise
  disjoint subsets of
  $P$ satisfying the conditions Definition~\ref{defrot}, i.e., 
\begin{itemize}
\item $A$ is a downset,
\item $C$ is an up-set,
\item $A<C$,
\item $A\cup B\cup C=P$.
\end{itemize}
Then we call the triple $A,B,C$ an \emph{extendible triple}.
\end{definition} 

\subsection{Decomposing rotations}

Switching a graph with respect to a finite subset $X$ of its vertices can be done by switching the graph consecutively with respect to $\{x\}$ for all $x\in X$ in any order. It turns out that we have a similar
phenomenon for cuts and rotations, although we have to be a  bit more
careful.

\begin{proposition}\label{rotprop}
Let $A,B,C\subseteq P$ be an extendible triple of a poset $\P=(P,\leq)$,
 and let $E\subseteq F\subseteq P$ be downsets. Then the following hold:
\begin{itemize}
\item[(1)] $\Ro_{A,C}= \Cu_{A\cup B}\Cu_A  $,
\item[(2)] $\Cu_{F}= \Cu_{F\setminus E}\Cu_{E}  $.
\end{itemize}
Moreover,
if $A$ is finite, and $a_1,\ldots,a_k\in A$ enumerate $A$ in such a way that $a_j\not < a_i$ for all $1\leq i<j\leq k$, then 
\begin{itemize}
\item[(3)] $\Cu_{A}= \Cu_{a_k} \Cu_{a_{k-1}}\cdots \Cu_{a_2} \Cu_{a_1}  $.
\end{itemize}

\end{proposition}

\begin{proof}
The first part of the statement follows from the definition observing
that the relationship between the elements of $A$ and $C$ changes
twice by cuts, hence their ordering is reversed. The elements of $B$
alter their relationship to $A$ and $C$ once, satisfying
Definition~\ref{defrot}. Item~(2) is shown similarly. Item~(3) follows from the iterated use of
(2): observe that 
$\{a_i\}$ is a one-element downset in $ \Cu_{a_{i-1}}\cdots \Cu_{a_2} \Cu_{a_1}(\P)$.

\end{proof}

\begin{corollary}\label{rotcut}
Any rotation is the composition of two
cuts. In a finite poset any cut or rotation can be obtained as a
composition of cuts with respect to 
one-element sets.
\end{corollary}


\subsection{Rotation-equivalence}

Similarly to  the Seidel-switch for graphs,  rotations define a partition of the finite posets on the same domain.

\begin{definition}
We say that two posets $(P,\leq)$ and $(P,\preceq)$ are
\emph{rotation-equivalent} iff $(P,\leq)$ can be mapped to  $(P,\preceq)$ by
using a series of rotations. \emph{Cut-equivalence} is defined similarly. 
\end{definition} 

\begin{proposition}
Rotation-equivalence and cut-equivalence are identical relations.
\end{proposition}
\begin{proof}
This follows from Corollary~\ref{rotcut}.
\end{proof}

\begin{proposition}\label{prop:equivalence}
Rotation-equivalence is an equivalence relation.
\end{proposition}

\begin{proof}
 Reflexivity is implied by $\Ro_{\emptyset,\emptyset}(\P)=\P$ for every
 poset $\P$. Transitivity is a consequence of the definition. For
 symmetry, let $\P=(P,\leq)$ and $\Q=(P,\preceq)$ be two posets, and suppose first that  for some rotation $\Ro_{A,C}$ of $\P$, $\Ro_{A,C}(\P)$ is isomorphic to $\Q$. Assume without loss of generality that $\Ro_{A,C}(\P)=\Q$. 
 Then $A$ is an up-set and $C$ is a
 downset in $\Q=\Ro_{A,C}(\P)$, and $C\prec A$. Hence  $\Ro_{C,A}(\Q)$ is defined and equal to $\P$. Now if $\Q$ can be obtained from $\P$ applying more than one rotation, then every step can be reversed as in the above argument.
 \end{proof}


Rotations divide the three-element posets with domain $\{a,b,c\}$ into the following three equivalence classes.

\begin{itemize}
\item[$\mathcal{O}_1$:] the class of  the 3-element
  antichain $a\perp b, b\perp c, c\perp a$;

$a<b,a<c, b  \perp c $; \quad $b<a,b<c, a  \perp c $; \quad $c<a,c<b, b  \perp c $;

$a>b,a>c, b  \perp c $; \quad $b>a,b>c, a  \perp c $; \quad $c>a,c>b, b  \perp c $; 

\item[$\mathcal{O}_2$:] the class of a 3-element chain: 

$a<b<c$;\quad $b<c<a$;\quad $c<a<b$;

$a <b, c\perp a, c\perp b$;\quad
$b <c, a\perp b, a\perp c$;\quad
$c <a, b\perp a, b\perp c$;
\item[$\mathcal{O}_3$:] the dual of $\mathcal{O}_2$: 

$a>b>c$; \quad $b>c>a$; \quad $c>a>b$;

$a >b, c\perp a, c\perp b$;\quad
$b >c, a\perp b, a\perp c$;\quad
$c >a, b\perp a, b\perp c$.
\end{itemize}

Interestingly, the rotation classes of the three-element subsets of a finite partial order, together with its maximal elements, determine the whole partial order.

\begin{proposition}\label{thm:maxdet}
Let $(P,\leq)$ and $(P,\preceq )$ be two finite posets. 
Then the following are equivalent:
\begin{enumerate}
\item $(P,\leq)= (P,\preceq )$.
\item $\max (P,\leq)=\max
(P,\preceq )$, and   $(P,\leq)$ and $(P,\preceq)$ are rotation-equivalent.
\item $\max (P,\leq)=\max
(P,\preceq )$, and   
 for all $a,b,c\in P$ the posets $(\{a,b,c\},\leq)$ and ${(\{a,b,c\},\preceq )}$  are
rotation-equivalent. 
\end{enumerate}
\end{proposition} 

\begin{proof} 
The implications (1) $\implies$ (2) and (2) $\implies$ (3) are obvious. To see that (3) implies (1), let $M=\max (P,\leq)=\max(P,\preceq )$. If $M=P$ then both posets
are antichains and the statement holds.
At first we examine the  relationship of the maximal elements to the
other elements of the poset.  If $|M|=1$, then we have a unique
maximal element $m$, and  $a\leq m$ for 
every $a\in P$. Now, let $|M|\geq 2$, $a\in P\setminus M$ and  $m_1$
and $m_2$ be two  distinct elements of  $M$. Let us assume that
$(a,m_1,m_2)\in \mathcal{O}_2$. As $m_1\perp m_2$,  the triple
$(a,m_1,m_2)$ is not a chain. The three element antichain is not in
$\mathcal{O}_2$, hence there is a comparability among the elements of 
$\{a,m_1,m_2\}$. By the maximality of $m_1$ and $m_2$ either $a<m_1$
or $a<m_2$. Only the first case can happen in $\mathcal{O}_2$, thus
we must have 
$a<m_1$ and $a\perp m_2$. Similarly, if $(a,m_1,m_2)\in
\mathcal{O}_3$, then $a<m_2$ and $a\perp m_1$. Finally assume that
$(a,m_1,m_2)\in \mathcal{O}_1$. Then either $a<m_1,m_2$ or $a\perp m_1,m_2$.
 The element
$a$ is below at least one maximal element, say  $a<m$. If $(a,m,m')\in
 \mathcal{O}_1$  for every $m'\in M$, then $a<m'$ for every $m'\in M$. Now we
can determine the relationship between the elements of $M$ and $P$: if 
$(a,m_1,m_2)\in \mathcal{O}_1$ for every $m_1,m_2\in M$, then $a<m$
for every $m\in M$. Otherwise, $a<m$ for some $m\in M$ if and only if
there is an element $m'\in M$ such that $(a,m,m')\in \mathcal{O}_2$.

Now let $a,b\in P\setminus M$. Let us choose an $m\in M$ satisfying
$b<m$. As $a\not> m$, we have that $a<b$ if and only if 
$(a,b,m)\in\mathcal{O}_2$.

We have shown that the maximal elements and the rotation classes of
the 3-element subsets uniquely determine $\leq$. They determine $\preceq$ in
the same way.
Hence, $(P,\leq)=(P,\preceq )$. 
\end{proof}

We remark that the equivalence of (1) and (2) of this proposition does not hold for infinite posets. Consider the order of the rationals $(\mathbb{Q},\leq)$ and apply the rotation $\Ro_{A,C}$ to this poset, where $A$ is the downset of negative numbers and $C$ its complement: while Item~(2) of Proposition~\ref{thm:maxdet} holds, (1) fails. We shall see in Section~\ref{sect:randompo} that the equivalence of (2) and (3) also holds for infinite posets.


\begin{corollary}\label{orbitsize}
Let $(P,\leq )$ be a finite poset. Then there are at most  $2^{|P|}$ 
posets with domain $P$ which are rotation-equivalent to  $(P,\leq )$.
\end{corollary}
\begin{proof} Any poset with domain $P$ in the  rotation-equivalence 
class  of  $(P,\leq)$ is  determined by its set of maximal elements by Proposition~\ref{thm:maxdet}. There  
are at most $2^{|P|}$ many choices to the set of the maximal elements.
\end{proof}

Note that Corollary~\ref{orbitsize} is a little misleading. In
most cases posets are considered  up to isomorphism; the number of
non-isomorphic rotation-equivalent posets can be much less. For example, we shall see below that in the class of
an $n$-element antichain there are $2^n$ posets, but only a set of $n+1$
non-isomorphic posets.

\begin{problem} Find the possible sizes of rotation-equivalence classes
  of the posets of size $n$. In particular, what is the maximum and minimum size
  of such a class? 
\end{problem}


We will now have a look at the rotation-equivalence classes of finite chains and antichains. In the following, let $\mathcal{C}_n$ denote the $n$-element
chain $1<2<\dots<n$ and $\mathcal{AC}_n$ denote the $n$-element
antichain on $\{1,\dots, n \}$. 

\begin{proposition}\label{chain} A finite poset $\P$ with domain $\{1,\ldots,n\}$ is rotation-equivalent
to 
\begin{enumerate}
\item $\mathcal{C}_n$ iff $\P$ is isomorphic to the disjoint union of two
  (possibly empty) chains $\mathcal{C}_s$ and $\mathcal{C}_t$, where $s+t=n$. 

\item $\mathcal{AC}_n$ iff $\P$ is isomorphic to a linear sum of 
 two   (possibly empty) antichains $\mathcal{AC}_s\oplus
 \mathcal{AC}_t$,   where $s+t=n$.
\end{enumerate}
\end{proposition}

\begin{proof}
We prove (1). Clearly, the disjoint union $\D$ of $\mathcal{C}_s$ and $\mathcal{C}_t$, where $s+t=n$, is rotation-equivalent to $\C_n$: just apply $\Cu_A$ to $\C_n$, for an $s$-element downset $A$ of $\C_n$. For the converse, by Proposition~\ref{rotprop} it suffices to verify that applying a cut with respect to a minimal element of $\D$, we again obtain a poset which is a disjoint union of two chains. Let $a$ be a minimal element of $\D$, say $a\in\C_s$. Then $\Cu_a(\D)$ is the disjoint union of $\mathcal{C}_{s-1}$ and $\mathcal{C}_{t+1}$, because $a$ is removed from $\C_s$ and added on top of $\C_t$.

The argument for (2) is equally simple and left to the reader.
\end{proof}

\begin{lemma}\label{max} Let $\P=(P,\leq)$ be a poset and
  $S,T\subseteq P$ be 
  antichains with  $S<T$ and such that $S\cup T\ne\emptyset$. Then there is a rotation  $\Ro_{A,C}$ of $\P$  such that
 $\max(\Ro_{A,C}(\P))=S\cup T$. 
\end{lemma}
\begin{proof} 
Let $S $ and $T$ be given. By replacing $S$ by $T$ and $T$ by $\emptyset$ in case $S$ is empty, we may assume that $S$ is non-empty. Set
  $$A=\{x\in P\;|\; x\leq s \text{ for some }   s\in S \}$$ and
$$ C=\{x\in P\;|\; x> S  \text{ and } 
x\not\leq t    \text{ for any }   t\in T  \}.$$
Since $S\ne \emptyset$ we have $A\ne\emptyset$. Setting $B=P\setminus A\cup C$, and hence it follows from the definition of a rotation that 
$$\max(\Ro_{A,C}(\P))= \max(A,\leq)\cup \{b\in \max(B,\leq)\,|\,b>A   \}. $$ 
It remains to show that this set equals $S\cup T$. Clearly, $ \max(A,\leq)=S$; we show $T= \{b\in \max(B,\leq)\,|\,b>A   \}$. Let $t\in T$. Then $t\notin A$ since $S<T$; moreover, $t\notin C$ by the definition of $C$, so $t\in B$. No other element of $B$ can be above $t$ by the definition of $C$, hence $t\in\max(B,\leq)$. 
    Since $t> S$ we derive $t\in \{b\in \max(B,\leq)\,|\,b> A   \}$. For the converse, let $b$ be an element of the latter set; then since $b>S$ but $b\notin C$, we have $b\leq t$ for some $t\in T$. We already know $t\in B$, and so $b=t\in T$.
\end{proof}


\begin{theorem}\label{cor:roteq}
 Let $(P,\leq)$ and $(P,\preceq )$ be finite posets. Then
 the following are equivalent:
\begin{enumerate}
 \item $(P,\leq)$ and $(P,\preceq )$ are  rotation-equivalent.    
\item There is a rotation mapping  $(P,\leq)$ to $(P,\preceq )$.
\item For all $a,b,c\in P$ the
 posets $(\{a,b,c\},\leq)$ and $(\{a,b,c\},\preceq )$
 are rotation-equivalent.
\end{enumerate} 
\end{theorem}

\begin{proof} 
The implication from (2) to (1) is trivial, and (1) $\implies$ (3) is obvious. To see  $(3)\implies (2)$, let $M=\max (P,  \preceq )$. We claim that $M$ is a linear sum of at most two
  antichains in  ${(P,\leq)}$. To see this, we distinguish two cases. If $M$ is an antichain in $(P,\leq)$ then there is nothing to show. Otherwise, pick $a,b\in M$ such that $a< b$. Now let $c\in M$ be arbitrary. Then, since $\{a,b,c\}$ forms an antichain in $(P,  \preceq )$ and since $(\{a,b,c\},\leq)$ and $(\{a,b,c\},\preceq )$
 are rotation-equivalent, $c$ must be comparable to precisely one element in $\{a,b\}$, say without loss of generality $a$. Then $a<c$ and $b\perp c$ in ${(P,\leq)}$. It follows that $M$ is the linear sum of two antichains in ${(P,\leq)}$, namely those elements of $M$ which are incomparable with $a$ and those which are incomparable with $b$. 
By Lemma~\ref{max} there is a  rotation that maps $(P,\leq)$
to  a poset $(P, \leq')$ such that $M={\max(P,\leq')}$. We have that the three-element subsets of  $(P, \leq')$ and $(P,\preceq )$ are rotation-equivalent, and so Proposition~\ref{thm:maxdet} implies $(P, \leq')=(P,\preceq )$, proving (2).
\end{proof}

\begin{corollary}\label{rotrot=rot}  If 
$(P,\leq)$ and  $(P,\preceq )$ are rotation-equivalent finite posets, then
  there exists a single rotation which maps 
$(P,\leq)$ to $(P,\preceq )$. Consequently, the composition of two rotations is
  a rotation.
\end{corollary}

\begin{proof} The statement follows from Theorem~\ref{cor:roteq}.
\end{proof}

We now show that the decision problem whether or not two given finite posets
are isomorphic and the decision problem whether or not these posets are rotation-equivalent have the
same computational complexity.

\begin{theorem}\label{compl} It is poset-isomorphism complete to
  decide for two given finite posets $\P,\Q$ whether or not $\P$ is rotation-equivalent to a poset isomorphic to $\Q$.
\end{theorem}

\begin{proof} 
At first we reduce the poset isomorphism problem to the rotation-equivalence problem. Let $\P=(P,\leq)$ and $\Q=(Q,\preceq)$ be two finite posets. Since isomorphic posets have the same size, we 
may assume that $|P|=|Q|=n$. Let $\S$ be an $n$-element antichain on some set $S$ which is disjoint from $P$ and $Q$. Let $\P'$ and $\Q'$ be the disjoint unions of $\P$ and $\Q$ with $\S$. To prove the theorem it is enough to show that $\P$ and $\Q$ are isomorphic if and only if $\P'$ and $\Q'$ are rotation-equivalent. One direction being trivial, we assume that $\P'$ and $\Q'$ are rotation-equivalent, i.e., there exists a rotation $\Ro_{A,C}$ of $\P'$ such that $\Ro_{A,C}(\P')$ is isomorphic to $\Q'$. Suppose that $A\cup C\neq\emptyset$. If $A\cap S\neq \emptyset$, then the symmetric closure of the poset relation of $\Ro_{A,C}(\P')$ is connected, a contradiction since the latter is not the case for $\Q'$. The same contradiction arises when $C\cap S\neq \emptyset$, so we must have $A\cap S=C\cap S=\emptyset$. In that case, all elements of $S$ belong to the same connected component of symmetric closure of the poset relation of $\Ro_{A,C}(\P')$, hence there are at most $n$ such components, again a contradiction. Hence, $A\cup C=\emptyset$, and so $\P'$ and $\Q'$ are in fact isomorphic. But then $\P$ and $\Q$ are isomorphic as well.

We now reduce the rotation-equivalence problem to the isomorphism problem. Let $\P=(P,\leq)$ and $\Q=(Q,\preceq)$ be two finite posets. Pick any $p\in P$; by Lemma~\ref{max}, there exists a unique poset $\P_p$ on $P$ which is rotation-equivalent to $\P$ and whose only maximal element is $p$. Similarly, for every $q\in Q$ there exists a unique poset $\Q_q$ on $Q$ which is rotation-equivalent to $\Q$ and whose only maximal element is $q$. By Proposition~\ref{thm:maxdet}, $\P$ and $\Q$ are rotation-equivalent if and only if $\P_p$ is isomorphic to $\Q_q$ for some $q\in Q$.
\end{proof}

In the following proposition we only observe that poset-isomorphism completeness is equivalent to graph-isomorphism completeness.

\begin{proposition}
It is graph-isomorphism complete to
  decide for two given finite posets $\P,\Q$ whether or not $\P$ and $\Q$ are isomorphic.
\end{proposition}
\begin{proof}
Let $(V,E)$ and $(W,F)$ be two finite graphs. Let $V'$ be the set of one- and two-element subsets of $V$. Define a poset $\P=(V',\leq)$ by setting $\{u\}\leq \{u\}$, $\{u\}\leq\{u,v\}$ iff $(u,v)\in E$ and $\{u\}\geq\{u,v\}$ iff $(u,v)\notin E$. Define $\Q$ from $(W,F)$ is the same manner. Then $\P$ and $\Q$ are isomorphic if and only if $(V,E)$ and $(W,F)$ are isomorphic. This reduces the graph-isomorphism problem to the poset-isomorphism problem.

For the converse, let $(P,\leq)$ and $(Q,\preceq)$ be two finite posets. We may assume that $|P|=|Q|=n$. Let $s$ be the number of levels of $P$.
For every level $1\leq i\leq s$, pick elements $p^i_1,\ldots,p^i_{n+i}$ outside $P$, and define a graph on $P':=P\cup\{p^i_j\; |\; 1\leq i\leq s\ \wedge\  1\leq j\leq n+i\}$ as follows. All $p^i_j, p^i_{k}$ are connected by an edge; moreover, all $p^i_j$ are connected with all elements of the $i$-th level of $(P,\leq)$. We connect furthermore $x,y\in P$ by an edge iff they are comparable with respect to $\leq$, and they are in adjacent levels. Define a graph from $(Q,\preceq)$ in the same manner. Then the two graphs obtained are isomorphic if and only if $(P,\leq)$ and $(Q,\preceq)$ are isomorphic.
\end{proof}

We finish this section with an extension of our results on finite posets to infinite posets via compactness. We would like to remark that we do not dispose of a ``direct'' proof of the following proposition which does not make use of our results on finite posets.

\begin{proposition}\label{prop:koenig}
Theorem~\ref{cor:roteq} and hence also Corollary~\ref{rotrot=rot} hold for all infinite posets $(P,\leq)$ and $(P,\preceq)$ as well.
\end{proposition}
\begin{proof}
We only need to show that (3) implies (2) in Theorem~\ref{cor:roteq}. Let $\{p_i\,|\,i\in\omega\}$ be an enumeration of $P$, and set $P_n:=\{p_i\,|\,i\leq n\}$, for all $i\in\omega$. By Theorem~\ref{cor:roteq}, there exist rotations $\Ro_n$ such that $\Ro_n(P_n,\leq)=(P_n,\preceq)$. Let $S$ consist of all restrictions of some $\Ro_n$ to some set $P_m$, where $m\leq n$. For $\Ro, \Ro'\in S$, set $\Ro\sqsubseteq \Ro'$ iff the domain of the poset rotated by $\Ro$ is contained in the corresponding domain for $\Ro'$, and $\Ro'$ agrees with $\Ro$ on the set where they are both defined. Then $\sqsubseteq$ defines a finitely branching tree on $S$. By K\"{o}nig's tree lemma, this tree has an infinite branch. This branch defines a rotation which sends $(P,\leq)$ to $(P,\preceq)$.
\end{proof}

\section{The random poset}\label{sect:randompo}
 
In this section we turn our attention to rotations of the random poset $\pp=(P,\leq)$. Up to isomorphism, $\pp$ is the unique countably infinite partial order satisfying the following \emph{extension property}:

\begin{itemize}
\item[(EXT)] For every finite subset $Q\subseteq P$ and every one-point extension $(Q\cup\{s\},\leq)$ of $(Q,\leq)$
 there exists $a\in P$ such that $(Q\cup\{s\},\leq)$ and $(Q\cup\{a\},\leq)$ are isomorphic via the mapping between these orders which fixes all $q\in Q$ and sends $s$ to $a$.
\end{itemize}

The image of a
rotation of $\pp$ is not necessarily isomorphic to $\pp$. For
example, if $A=\{x\in \pp\,|\, x\leq a   \}$ for some $a\in \pp$ and $C>A$ is arbitrary, then
$\Ro_{A,C}(\pp)$  has a maximal element (namely the element $a$), hence it cannot be isomorphic to
the random poset. In the companion paper~\cite{ppppsz}, we showed the existence of non-trivial rotations of $\pp$ which send $\pp$ to a partial order isomorphic to $\pp$ using model-theoretic methods; here, we give a combinatorial proof of this fact using one-point extensions. The existence of such a rotation of $\pp$ implies that the group of rotating permutations in Theorem~\ref{reductok} really is a proper supergroup of $\aut(\pp)$: for, if $\Ro_{A,C}(\pp)\simeq \pp$ for $A,C\subseteq \pp$ for which either $A$ or $C$ is non-empty, then this isomorphism is witnessed by a permutation in $\sym(\pp)\setminus\aut(\pp)$, and this permutation separates the group of rotating permutations from $\aut(\pp)$.

\begin{theorem}\label{thm:constr}
There is a rotation of the random poset $\pp$ whose image is
isomorphic to $\pp$.
\end{theorem} 
\begin{proof} 
Let $a\in \pp$ and  let  $A=\{x\in \pp\,|\, x<a     \}$,   $C=\{x\in \pp\,|\, x>a
  \}$ and $B=\{x\in \pp\,|\, x\perp a \}$.  Then the rotation $\Ro_{A,C}$ is well-defined.
 Now omit $a$ from
  $\pp$. Then $\pp\setminus\{a\}$ is isomorphic to $\pp$ -- this is well-known and easy to see by verifying the extension property.
  Thus we can
  denote  $\pp\setminus\{a\}$ by  $\pp$.
Then  $A\cup B\cup C=\pp$ and  $\Ro_{A,C}$ is a rotation of $\pp$. Let
$\mathcal Q=(P, \preceq):=\Ro_{A,C}(\pp)$. We
show that $\qq$ is isomorphic to $\pp$. 

To do this, we check that (EXT) holds for $\qq$. Let $S\subseteq P$ be finite and $A',B',C'\subseteq S$ be such that $A'\prec C'$, $A'$ is a $\preceq$-downset,
$C'$ is an $\preceq$-up-set in $S=A'\cup B' \cup C'$. Then the
following  relations hold:
\begin{alignat}{3}
(C\cap C')&>&(A\cap C')\cup(B\cap B')\cup(C\cap A')\\
(B\cap C')\cup(C\cap B')&>&(A\cap B')\cup(B\cap A')\\
(A\cap C')\cup(B\cap B')\cup(C\cap A')&>&(A\cap A')
\end{alignat}
For Item~(1), $C\cap C'>A\cap C'$ holds because $A<C$.
In $\mathcal Q$  we have $A'\prec C'$, and $\Ro_{A,C}$  does not change
the relations inside $C$, hence  $C\cap C'> C\cap A'$. For $c\in
C$ and $b\in B$ either $b<c$ or $b\perp_< C$. Thus either
$b\perp_\prec c$ or  $b\succ c$. In $\mathcal Q$ for any  $c\in
C'$ and $b\in B'$ either $b\prec c$ or $b\perp_\prec c$. Thus for any 
      $b\in B\cap B'$ and  $c\in C\cap C'$ we must have
$b<c$. Hence (1) holds.
Item (3) is proved dually. 

For Item~(2), $C\cap B'>A\cap B'$ because $A<C$, and
$(B\cap C')>(B\cap A')$ holds because $A'\prec C'$.
 For $y\in
A$ and $x\in B$ either $y<x$ or $y\perp_< x$. Thus either
$y\perp_\prec x$ or  $x\prec y$.
 In $\mathcal Q$ for any  $x\in
C'$ and $y\in B'$ either $y\prec x$ or $y\perp_\prec x$. Thus for any 
      $x\in B\cap C'$ and  $y\in A\cap B'$ we must have $x>y$, so
$B\cap C' >A\cap B'$. The case $C\cap B'> B\cap A'$ is proved
similarly.

At first, let $C\cap C'\ne\emptyset$. Now, $C'$ is an up-set in $S$ with respect to $\preceq$.
As $C>A$ in $\pp$ we have $C\prec A$ in $\qq$. Thus $c\prec A$ for any
element $c\in C\cap C'$, so $C'\supseteq A$ in $\qq$. The subsets $A',B',C'$ 
are disjoint, hence  $A\cap B'=\emptyset$ and  $A\cap A'=\emptyset$.
  As $C $ is a $\prec$-downset and  $c\succ A'$ for any
element $c\in C\cap C'$, we have    $A'\subset
C$. Thus, as    the subsets $A,B,C$ are disjoint, $A\cap A'=B\cap
A'=\emptyset$,   as well. Now let

\begin{alignat*}{3}
X_0=&(A\cap A')\cup(B\cap C')\cup(C\cap B')&=&(B\cap C')\cup(C\cap B'),\\
X_1=&(A\cap B')\cup(B\cap A')\cup(C\cap C')&=&(C\cap C'),\\
X_2=&(A\cap C')\cup(B\cap B')\cup(C\cap A').
\end{alignat*}

We claim that ${X_2,X_0,X_1}$ is an extendible triple with respect to the order $\leq$, and there is a
  $q\in C$ extending it. Moreover, $q$ extends $S$ in $\qq$.
At first, as $C$ and $C'$ are up-sets in $\pp$ and $(S,\preceq)$,
respectively, $X_1$ is an up-set in $(S,\leq)$. Secondly, we show that
$X_2$ is a downset in $(S,\leq)$. For this we need to show that no element of $X_2$
is above an element of $X_0$.

As $A,B,C$ is an extendible  triple in $\pp$,
considering the relationships of the elements in $\pp$ we have that 
$x\not>y$, if

\begin{itemize}

\item   $x\in (A\cap C')$ and $y\in (C\cap B')$, 
\item $x\in (B\cap B')$ and $y\in (C\cap B')$,
\item  $x\in (A\cap C')$ and $y\in (B\cap C')$.
\end{itemize}
If $x$ and $y$ are both in $A,B$ or $C$, then their relationship is
the same in $\pp$ and $\qq$. Hence $x\not>y$, if
\begin{itemize}
\item  $x\in (B\cap B')$ and $y\in (B\cap C')$,
\item   $x\in (C\cap A')$ and $y\in (C\cap B')$.
\end{itemize}
Finally, let 
\begin{itemize}
\item  $x\in (C\cap A')$ and $y\in (B\cap C')$.
\end{itemize}
Then $x\prec y$ and $x>y$ or $x\perp y$. Hence, by the definition of
the rotation $x\perp y$ holds. Thus $X_2$ is a downset in $(S,\leq)$ and  ${X_2,X_0,X_1}$ is an extendible triple with respect to $\leq$.

Now we show that  ${X_2,X_0,X_1}$ can be extended in $C$. Let us
include again the element $a$ that we have omitted from $\pp$. Now,
$\pp \cup \{a\}$ is isomorphic to $\pp$  and
 $X_2\cup \{a\}$ is a downset in $(S,\leq)$, $X_1$ is an up-set in $(S,\leq)$  and $X_2\cup\{a\}<X_1$.
As $\pp\cup\{a\}$
is the random poset, there exists a point $q\in \pp$ satifying the conditions
$X_2\cup\{a\}<q,X_0\perp q,X_1>q$. Hence, $q$ is in $C$ and $q$
extends  ${X_2,X_0,X_1}$ 
in $\pp$. 

We need to show that  $q$ extends $S$ in $\qq$.  The
relationship of $q$ to the elements of $C$ is not altered, hence they remain the same in $\Q$. If $x\in (B\cap
B')$, then $x>q$ turns to $x\perp_\prec q$.  If $x\in (A\cap
C')$, then $x<q$ turns to $x\succ q$. Finally, if $x\in (B\cap
C')$, then $x\perp q$ turns to $x\succ q$, hence  $q$  extends $S$ in $\qq$.

The dual case when $A\cap A'\ne \emptyset$ can be handled similary. In
this case

\begin{alignat*}{3}
X_0=&(A\cap A')\cup(B\cap C')\cup(C\cap B')&=&(A\cap A') ,\\
X_1=&(A\cap B')\cup(B\cap A')\cup(C\cap C')&=&(A\cap B')\cup(B\cap A')\\
X_2=&(A\cap C')\cup(B\cap B')\cup(C\cap A')&
\end{alignat*}
 ${X_0,X_1,X_2}$ will be the extendible triple with respect to $\leq$, and
we can find a  point $q\in A$ such that  $q$  extends $S$ in $\qq$.

Finally, assume that $A\cap A'=C\cap C'=\emptyset$. In this case

\begin{alignat*}{3}
X_0=&(A\cap A')\cup(B\cap C')\cup(C\cap B')&=&(B\cap C')\cup(C\cap B') ,\\
X_1=&(A\cap B')\cup(B\cap A')\cup(C\cap C')&=&(A\cap B')\cup(B\cap A')\\
X_2=&(A\cap C')\cup(B\cap B')\cup(C\cap A').&
\end{alignat*}
 We are going to prove that
 ${X_1,X_2,X_0}$  is an extendible triple with respect to $\leq$, and
we can find a  point $q\in B$ such that  $q$  extends $S$ in $\qq$.
The arguments from the first two cases imply that $X_0$ is an up-set
and $X_1$ is a downset in $(S,\leq)$.
Consider  $\pp\cup\{a\}$, as before. The triple
$X_1,X_2\cup\{a\},X_0$ is an extendible triple with respect to $\leq$. Hence, there exists
$q\in B$ extending it in $\pp$. Now,   $q$ extends $S$
in $\qq$. It can be checked similarly to the previous  cases:  The
relationship of $q$ to the elements of $B$ is not altered at
$q$, hence they remain the same in $\Q$. 
 If $x\in (A\cap
B')$, then $x<q$ turns to $x\perp_\prec q$.  If $x\in (C\cap
B')$, then $x>q$ turns to $x\perp_\prec q$. Finally, if $x\in (A\cap
C')$, then $x\perp q$ turns to $x\succ q$, and if  $x\in (C\cap
A')$, then $x\perp q$ turns to $x\succ q$. Hence  $q$  extends $S$ in $\qq$.

We obtained that (EXT) holds for $\qq$. Therefore $\qq$ is isomorphic
to the random poset. 
\end{proof}

We now show that any rotation of a finite poset can be interpreted as the restriction of a rotating permutation of $\pp$, showing that our notion of rotation of a finite poset really is the analogue of the Seidel-switch for posets.

\begin{proposition}\label{prop:finiterotationsarerestrictions}
Let $(X,\leq)$ be a finite poset, and let $\Ro_{U,V}$ be a rotation of $(X,\leq)$. Let $e_1: (X,\leq) \to \pp$ and $e_2: \Ro_{U,V}(X,\leq) \to \pp$ be embeddings. Then there exists a rotating permutation $\alpha$ of $\pp$ such that $e_2=\alpha\circ e_1$.
\end{proposition}
\begin{proof}
Set $A':=e_1[U]$, $C':=e_1[V]$, and $B':=e_1[X]\setminus A'\cup C'$. By the extension property of $\pp$, there exists $a\in P$ such that $A'< a$, $a< C'$, and $a\perp B'$. 
 Set $A=\{x\in \pp\,|\, x<a     \}$,   $C=\{x\in \pp\,|\, x>a
  \}$ and $B=\{x\in \pp\,|\, x\perp a \}$. By the proof of Theorem~\ref{thm:constr}, $(P\setminus\{a\},\leq)$ and $\Ro_{A,C}(P\setminus\{a\},\leq)$ are isomorphic; let $\beta$ be a rotating permutation of $(P\setminus\{a\},\leq)$ witnessing this. Pick any isomorphism $i: \pp\to (P\setminus\{a\},\leq)$. Then $\gamma:=i^{-1}\circ \beta\circ i$ is a rotating permutation of $\pp$ for the rotation $\Ro_{i^{-1}[A],i^{-1}[C]}$. Since $A'\subseteq A$, $C'\subseteq C$, and $B'\subseteq B$, we have that $i^{-1}\circ \beta\circ e_1$ changes the relations between the elements of $(X,\leq)$ just like the rotation $\Ro_{U,V}$. Thus, by the homogeneity of $\pp$, there exists $\delta\in\aut(\pp)$ such that $\delta\circ i^{-1}\circ \beta\circ e_1=e_2$. Again by the homogeneity of $\pp$, there exists $\varepsilon\in\aut(\pp)$ such that $i^{-1}\circ e_1=\varepsilon\circ e_1$. Set $\alpha:=\delta\circ\gamma\circ \varepsilon$. Composed of rotating permutations and automorphisms of $\pp$, $\alpha$ is itself a rotating permutation of $\pp$. Moreover, $\alpha\circ e_1=\delta\circ\gamma\circ \varepsilon\circ e_1= \delta\circ i^{-1}\circ \beta\circ i\circ i^{-1}\circ e_1=e_2$, proving the proposition.
\end{proof}

The remainder of this paper is devoted to showing that the rotating permutations of $\pp$ can be described as the automorphisms of 
a homogeneous structure with one ternary relation. In the following, for $i\in\{1,2,3\}$ we identify $\mathcal{O}_i$ with the ternary relation on $\pp$ which consists of all triples in $\pp$ which induce a poset isomorphic with a poset in $\mathcal{O}_i$. The following has already been observed in~\cite{ppppsz}, but here we provide a much shorter proof which draws on our results about finite rotations.

\begin{proposition}
The rotating permutations of $\pp$ are precisely the automorphisms of the structure $(P,\mathcal{O}_1,\mathcal{O}_2,\mathcal{O}_3)$.
\end{proposition}
\begin{proof}
Clearly, if $\alpha$ is a rotating permutation of $\pp$, then $(\{a,b,c\},\leq)$ and $(\{\alpha(a),\alpha(b),\alpha(c)\},\leq)$ are rotation-equivalent for all $a,b,c\in P$, and hence $\alpha$ is an automorphism of $(P,\mathcal{O}_1,\mathcal{O}_2,\mathcal{O}_3)$. Conversely, if a permutation $\alpha$ has the latter property, then setting $\alpha(x)\preceq \alpha(y)$ if and only if $x\leq y$ we get that $(P,\preceq)$ and $\pp$ satisfy condition~(3) of Theorem~\ref{cor:roteq}; hence, by Proposition~\ref{prop:koenig} there is a rotation $\Ro$ such that $\Ro(\pp)=(P,\preceq)$. By the definition of $\preceq$, the permutation $\alpha$ is an isomorphism from $\pp$ to $(P,\preceq)$, and hence it is a rotating permutation with respect to the rotation $\Ro$.
\end{proof}

As another application of our results, we shall see that $(P,\mathcal{O}_1,\mathcal{O}_2,\mathcal{O}_3)$ is homogeneous. This is a very strong property for a structure to have and implies many nice other properties. Among these is \emph{quantifier elminination}, i.e., every first-order formula over a homogeneous  structure is equivalent to a formula without quantifiers. We remark moreover that although $\pp$ is homogeneous and $(P,\mathcal{O}_1,\mathcal{O}_2,\mathcal{O}_3)$ has a first-order definition in $\pp$, it does not automatically follow that $(P,\mathcal{O}_1,\mathcal{O}_2,\mathcal{O}_3)$ is homogenous itself or first-order interdefinable with a homogeneous structure in a finite language: there exist counterexamples in similar situations (see the introduction of~\cite{thomas}).

\begin{theorem}\label{thm:homo}
$(P,\mathcal{O}_1,\mathcal{O}_2,\mathcal{O}_3)$ is homogeneous.
\end{theorem}
\begin{proof}
Let $S,T\subseteq P$ be finite and let $i:S\to T$ be a partial isomorphism between the structures induced by these sets in $(P,\mathcal{O}_1,\mathcal{O}_3,\mathcal{O}_2)$. Pick any $s\in S$ and set $t:=i(s)$. Let $\Ro$ and $\Ro'$ be rotations of $(S,\leq)$ and $(T,\leq)$, respectively, with the property that $s$ and $t$ are the unique maximal elements of $\Ro(S,\leq)$ and $\Ro'(T,\leq)$. Let $u:\Ro(S,\leq)\to\pp$ and $v:\Ro'(T,\leq)\to\pp$ be embeddings. Then by Proposition~\ref{prop:finiterotationsarerestrictions} there exist rotating permutations $\alpha, \beta$ of $\pp$ whose respective restrictions to $S$ and $T$ are equal to $u$ and $v$. 
Now set $i':=v \circ i\circ u^{-1}$. Then $i'$ preserves $\mathcal{O}_1, \mathcal{O}_2, \mathcal{O}_3$, and sends the unique maximal element of $(u[S],\leq)$ to the unique maximal element of $(v[S],\leq)$. Thus, by Proposition~\ref{thm:maxdet}, it is an isomorphism between these posets, which by the homogeneity of $\pp$ extends to an automorphism $\gamma$ of $\pp$. Hence, $\beta\circ \gamma\circ \alpha^{-1}$ is an extension of $i$ to an automorphism of $(P,\mathcal{O}_1,\mathcal{O}_3,\mathcal{O}_2)$.
\end{proof}

Observe that $\mathcal{O}_3$ can by defined from $\mathcal{O}_2$ by  $(a,b,c)\in \mathcal{O}_3 \leftrightarrow (c,b,a)\in \mathcal{O}_2
$. Moreover, $\mathcal{O}_1$ can be defined from $\mathcal{O}_2$ since it is the complement of $\mathcal{O}_2\cup \mathcal{O}_3$ in $P^3$. Hence, the automorphism groups of the structures $(P,\mathcal{O}_2)$, $(P,\mathcal{O}_3)$, and $(P,\mathcal{O}_1,\mathcal{O}_2,\mathcal{O}_3)$ are all identical: they consist of the rotating permutations.

\begin{corollary}
$(P,\mathcal{O}_2)$ and $(P,\mathcal{O}_3)$ are homogeneous, and the automorphisms of any of these structures are precisely the rotating permutations of $\pp$.
\end{corollary}
\begin{proof}
We show homogeneity for $(P,\mathcal{O}_2)$; the argument for $(P,\mathcal{O}_3)$ is identical. Observe that the definition of $\mathcal{O}_3$  from $\mathcal{O}_2$ given above do not use quantifiers; in other words, for any finite $S\subseteq P$ we have that the triples of elements in $S$ which are elements of $\mathcal{O}_2$ determine those triples which are elements of $\mathcal{O}_3$. Hence, any partial isomorphism between finite induced substructures of $(P,\mathcal{O}_2)$ is a partial isomorphism between the structures induced in $(P,\mathcal{O}_2,\mathcal{O}_3)$. For the same reason we have that any partial isomorphism between finite induced substructures of $(P,\mathcal{O}_2)$ is a partial isomorphism between the structures induced in $(P,\mathcal{O}_1,\mathcal{O}_2,\mathcal{O}_3)$. By Theorem~\ref{thm:homo}, any such partial isomorphism extends to an automorphism of $(P,\mathcal{O}_1,\mathcal{O}_2,\mathcal{O}_3)$, which is also an automorphism of $(P,\mathcal{O}_2)$.
\end{proof}

\end{document}